\documentclass[12pt]{amsart}
 \usepackage{amssymb, amsthm, amsmath, mathrsfs}
\usepackage[all,cmtip]{xy}

\DeclareSymbolFont{largesymbols}{OMX}{yhex}{m}{n}
\DeclareMathAccent{\widehat}{\mathord}{largesymbols}{"62}

\newtheorem{theorem}{Theorem}[section]
\newtheorem{corollary}[theorem]{Corollary}
\newtheorem{lemma}[theorem]{Lemma}
\newtheorem{proposition}[theorem]{Proposition}
\newtheorem{question}[theorem]{Question}

\theoremstyle{definition}
\newtheorem{definition}[theorem]{Definition}
\newtheorem{remark}[theorem]{Remark}

\newtheorem{example}[theorem]{Example}

\DeclareMathOperator{\Ima}{im}

\newcommand{\mdim}{{\rm mdim}}

\newcommand{\Wdim}{{\rm Wdim}}

\newcommand{\ord}{{\rm ord}}
\newcommand{\id}{{\rm id}}

 \newcommand{\cB}{{\mathcal B}}
  
  \newcommand{\cD}{{\mathcal D}}
  \newcommand{\cF}{{\mathcal F}}

 \newcommand {\cU}{{\mathcal U}}
  \newcommand {\cV}{{\mathcal V}}
  \newcommand{\cW}{{\mathcal W}}

 \newcommand{\bR}{{\mathbb R}}
 
 \newcommand{\bZ}{{\mathbb Z}}

 \newcommand{\rM}{{\rm M}}

\begin{document}

\title{Conditional mean dimension}

\author{Bingbing Liang}

\address{\hskip-\parindent
B.L., Department of Mathematical Science, Soochow University, 
Suzhou 215006, China.\\
The Institute of Mathematics of the Polish Academy of Sciences,
ul. \'{S}niadeckich  8, Warsaw 00-656, Poland}
\email{bbliang@suda.edu.cn, bliang@impan.pl}


\date{June, 2021}

\begin{abstract}
We introduce some notions of conditional mean dimension for a factor map between two topological dynamical systems and discuss their properties. With the help of these notions, we obtain an inequality to estimate the mean dimension of an extension system. The conditional mean dimension for $G$-extensions are computed. We also exhibit some applications in the dynamical embedding problems.
\end{abstract}

\maketitle

\section{Introduction}

Let $\Gamma$ be a countable amenable group. By a {\it dynamical system} $\Gamma \curvearrowright X$, we mean a compact metrizable space $X$ associated with a continuous action of $\Gamma$. Suppose $\Gamma \curvearrowright Y$ is another dynamical system  and $\pi\colon X \to Y$ is a continuous $\Gamma$-equivariant surjective map, i.e. a {\it factor map} between $X$ and $Y$. In such a setting, we call $\Gamma \curvearrowright X$ an {\it extension system} and $\Gamma \curvearrowright Y$ a {\it factor system}.

Mean (topological) dimension is a newly-introduced dynamical invariant by Gromov \cite{Gromov99T}, which measures the average dimension information of dynamical systems based on the covering dimension for compact Hausdorff spaces. It plays a crucial role in the embedding problem of dynamical systems \cite{Gutman11, Gutman15, Gutman16, GLT16, GT14, LW00}.

Since each fiber $\pi^{-1}(y)$ is a closed subset of the ambient system $\Gamma \curvearrowright X$, taking advantage of the ambient action, we can also discuss the  mean dimension
$\mdim(\pi^{-1}(y), \Gamma)$ for the fiber $\pi^{-1}(y)$ (see Definition \ref{mdim def of closed sets}).
When computing the mean dimension of moduli spaces of Brody curves, Tsukamoto established an inequality to estimate the mean dimension of the extension system in terms of the mean dimension of the factor system \cite[Theorem 4.6]{Tsukamoto08}. Based on this inequality, he asked the following question \cite[Problem 4.8]{Tsukamoto08}:

\begin{question}\label{main interest}
For a factor map $\pi\colon X \to Y$, is it true that
$\mdim(X) \leq \mdim(Y)+\sup_{y \in Y} \mdim(\pi^{-1}(y), \Gamma)$?
\end{question}

Observe that as $\Gamma$ is trivial, the above inequality recovers as the classical Hurewicz's formula \cite[Theorem VI 7]{HW41}\cite[Theorem 1.12.4, 3.3.10]{Engelking95}:
$$\dim(X) \leq \dim(Y) +\sup_{y \in Y} \dim(\pi^{-1}(y)).$$
So Question \ref{main interest} can be regarded as a dynamical concern of the Hurewicz's inequality.

We can also consider Question \ref{main interest} in parallel with entropy theory.  Historically, for $\Gamma=\bZ$, in \cite{Bowen71}, Bowen estimated the topological entropy $h(X)$ of $\Gamma \curvearrowright X$ in terms of the entropy of fibers $h(\pi^{-1}(y), \Gamma)$ for $y \in Y$, i.e.
$$h(X) \leq h(Y)+\sup_{y \in Y} h(\pi^{-1}(y), \Gamma).$$
In  particular, this verifies a conjecture of \cite[Conjecture 5]{AKM65} concerning the entropy of a skew product system.
Later, some versions of conditional entropy $h(X|Y)$ relative to a factor $\Gamma \curvearrowright Y$ were introduced and the related variational principles were established \cite{Misiurewicz76, DS02, Yan15}. In particular, it is shown that
$$h(X|Y)=\sup_{y \in Y} h(\pi^{-1}(y), \Gamma)$$
in the case $\Gamma =\bZ$ \cite{DS02} and in the general case $\Gamma$ is amenable \cite{Yan15}. This brings us the third  motivation to study Question \ref{main interest}.

Motivated from these points of view, we introduce some conditional versions of mean dimension relative to a factor system and study their  properties. When the factor is trivial, these conditional mean dimensions recover as the mean dimension.

In Section 2,  we first define the  conditional mean (topological) dimension. We study a class of extensions, called {\it $G$-extensions},  which generalize (topological) principal group extensions (Definition \ref{G-extension}). A $G$-extension is based on another  dynamical system $\Gamma \curvearrowright G$ (typically $G$ is a compact metrizable group and the action is by automorphisms). It turns out the notion of conditional mean dimension brings us a proper notion to strengthen some results on embedding problems of dynamical systems in terms of Rokhlin dimension of the factor systems.  We present such  applications in Theorem \ref{conditional embedding} and Corollary \ref{G-extension embedding}.  Under certain conditions, Question \ref{main interest} is confirmed.

In Section 3, in terms of notion of conditional mean dimension, we prove an inequality to estimate the mean dimension of an extension system. 

\begin{theorem} \label{subadditivity}
For any factor map $\pi\colon X \to Y$, we have
$$ \mdim(X)\leq \mdim(Y)+\mdim(X|Y).$$
\end{theorem}

The key technique of the proof is to approximate sufficiently large F{\o}lners set by smaller F{\o}lner sets as in the proof of \cite[Theorem 4.6]{Tsukamoto08}. 
In fact, we can slightly adjust the proof to obtain a more general inequality for a composition of two factor maps, or even the setting of the inverse limit of factor maps.

We remind the reader that $\mdim(X|Y)$ is an upper bound of $\sup_{y \in Y}\mdim( \pi^{-1}(y), \Gamma)$ (see Proposition \ref{fiber as lower bound}) and in some case they coincide with each other. Thus the above estimation can be thought as a weak version of the inequality in Question \ref{main interest}(see Proposition \ref{fiber inequality}). 

Note that as $X$ is a product system $Y\times Z$ for some dynamical system $\Gamma \curvearrowright Z$ associated with the diagonal action, $\pi$ is the projection map, then we have $\mdim(X|Y)=\mdim(Z)$ (see Proposition \ref{generalization}). This recovers as the subadditivity formula for Cartesian products of dynamical systems \cite[Proposition 2.8]{LW00}.

As a cousin of mean dimension, Lindenstrauss and Weiss introduced the metric mean dimension as an upper bound of mean dimension \cite{LW00}. This notion is a dynamical analogue of lower box dimension. In Section 4, we define the conditional version of the metric mean dimension. It is natural to ask whether the conditional metric mean dimension is an upper bound of conditional mean dimension (Question \ref{comparison}).

Downarowicz and Serafin introduced the topological fiber entropy given a measure on the factor system \cite[Definition 8]{DS02}. Motivated by this approach, we introduce an analogue for mean dimension. It turns that this  mean dimension given a measure serves as a lower bound for the conditional mean dimension (See Propositions \ref{inequalities} and \ref{equality}).

\noindent{\it Acknowledgements.}
I am grateful to the inspiring discussion with Hanfeng Li and the referee's valuable comment. The author is supported by the Research Support Funding GJ10700120 at Soochow University and the Institute of Mathematics of the Polish Academy of Sciences.

\section{Conditional mean topological dimension}

In this section, we define the notion of conditional mean topological dimension, discuss its properties, and compute some examples.

Let us first recall some machinery of amenable groups in the preparation of defining dynamical invariants.
For a  countable group $\Gamma$ denote by $\cF(\Gamma)$ the set of all nonempty finite subsets of  $\Gamma$.
\subsection{Amenable groups}
For each $K \in \cF(\Gamma)$ and $\delta > 0$, denote by $\cB(K, \delta)$ the set of all $F \in \cF(\Gamma)$ satisfying $|\{t \in F: Kt \subseteq F\}| \geq (1-\delta)|F|$. $\Gamma$ is called {\it amenable} if $\cB(K, \delta)$ is not empty for each pair  $(K, \delta)$.

The collection of pairs $(K, \delta)$ forms a net $\Lambda$ in the sense that $(K', \delta')\succeq (K, \delta)$ if $K'\supseteq K$ and $\delta'\leq \delta$. For a real-valued function $\varphi$ defined on $\cF(\Gamma) \cup \{\emptyset \}$, we say that {\it $\varphi(F)$ converges to $c\in \bR$ when $F\in \cF(\Gamma)$ becomes more and more
invariant}, denoted by $\lim_F\varphi(F)=c$, if for any $\varepsilon>0$ there is some $(K, \delta)\in \Lambda$ such that  $|\varphi(F)-c|<\varepsilon$ for all $F\in \cB(K, \delta)$.
In general, $\varlimsup_F \varphi(F)$ is defined as
$$\varlimsup_F \varphi(F): =\lim_{(K, \delta) \in \Lambda} \sup_{F \in \cB(K, \delta)} \varphi(F).$$

In the rest of this paper, $\Gamma$ will always denote a countable amenable group.
The following fundamental lemma, due to Ornstein and Weiss,  is crucial to define dynamical invariants for amenable group actions \cite[Theorem 6.1]{LW00}.
\begin{lemma} \label{OW}
Let  $\varphi \colon \cF(\Gamma)  \to [0, +\infty)$ be a map satisfying\\
\begin{enumerate}
    \item $\varphi(Fs)=\varphi(F)$ for all $F \in \cF(\Gamma)$ and $s \in \Gamma$; \\
    \item $\varphi(F_1\cup F_2) \leq \varphi(F_1) + \varphi(F_2)$ for all $F_1, F_2 \in \cF(\Gamma)$.\\
\end{enumerate}
Then the limit $\lim_F\varphi(F)/|F|$ exists.
\end{lemma}

\subsection{Conditional mean topological dimension}

Let $X$ be a compact metrizable space. For two finite open covers $\cU$ and $\cV$ of $X$, the joining $\cU \vee \cV $ is defined as
$\cU\vee \cV=\{U \cap V: U \in \cU, V \in V\}$. We say $\cU$ {\it refines} $\cV$, denoted by $\cU \succeq \cV$, if every element of $\cU$ is contained in some element of $\cV$. Denote by $\ord \ (\cU)$ the overlapping number of $\cU$, i.e.
$$\ord \ (\cU)=\max_{x\in X} \sum_{U \in \cU} 1_U(x)-1.$$

Now fix a factor map $\pi\colon X \to Y$ and a finite open cover $\cU$ of $X$.
Consider the number $\cD(\cU|Y):=\min (\ord \ (\cW))$ for $\cW$ ranging over all finite open covers  of $X$ such that $\{\pi^{-1}(y)\}_{y \in Y} \vee \cW$ {\it refines} $\cU$. We put $\cD(\cU):=\cD(\cU|Y)$ for $Y$ being a singleton.
For any $F \in \cF(\Gamma)$, denote by $\cU^F$ the finite open cover $\vee_{s \in F}s^{-1}\cU$.
To see the function  $\varphi\colon \cF(\Gamma)\cup \{\emptyset\} \to \bR$ sending $F$ to  $\cD(\cU^F|Y)$  satisfies the conditions of Lemma \ref{OW}, we have a conditional version of  \cite[Proposition 2.4]{LW00} to assist us.

\begin{lemma} \label{conditional bridge}
Suppose that $\pi: X \to Y$ is a continuous map and $\cU$ is a finite open cover of $X$. Then
$\cD(\cU|Y) \leq k$ if and only if there exists a continuous  map $f: X \to P$ for some polyhedron $P$ with $\dim (P) = k$ such that $\{f^{-1}(p)\cap \pi^{-1}(y)\}_{(p, y) \in P\times Y}$ refines $\cU$.
\end{lemma}

\begin{proof}
Firstly suppose that we have such a continuous map $f: X \to P$. Let $\varphi: X \to P \times Y$ be the map sending $x$ to $(f(x), \pi(x))$. By \cite[Proposition 2.4]{LW00}, there exists a finite open cover $\cV$ of $P \times Y$ such that $\varphi^{-1}(\cV)$ refines $\cU$. Without loss of generality, we may assume that $\cV$ is of a form $\cW \times \cV_1$ for some finite open cover $\cW$ of $P$ and finite open cover $\cV_1$ of $Y$ respectively. Then for every $W \in \cW$ and  $y \in Y$, we have $y \in V$ for some $V \in \cV_1$ and $\varphi^{-1}(W\times V) \subseteq U$ for some $U \in \cU$. Thus
$$f^{-1}(W)\cap \pi^{-1}(y) \subseteq f^{-1}(W)\cap \pi^{-1}(V)=\varphi^{-1}(W\times V) \subseteq U.$$
This concludes that $f^{-1}(\cW) \vee \{\pi^{-1}(y)\}_y$ refines $\cU$.
Choose a finite open cover $\cV_2$ of $P$ refining $\cW$ such that $\ord \ (\cV_2) \leq \dim (P)$. Then $\ord \ (f^{-1}(\cV_2)) \leq \ord \ (\cV_2) \leq \dim P$ and $f^{-1}(\cV_2) \vee \{\pi^{-1}(y)\}_y$ refines $\cU$. It follows that $\cD(\cU|Y) \leq \dim P =k$.

Now suppose that $\cD(\cU|Y) \leq k$.  By definition, there exists a finite open cover $\cW$ of $X$ such that $\{\pi^{-1}(y)\}_y \vee \cW$ refines $\cU$ and $\ord \ (\cW) \leq k$. Let $\{g_W\}_{W \in \cW}$ be a partition of unity subordinate to $\cW$ and $\Delta_\cW$ be the polyhedron induced from the  nerve complex of $\cW$. Define the map $g: X \to \Delta_\cW$ sending $x$ to $\sum_{W \in \cW} g_W(x)e_W$, where $e_W$ stands for the vertex indexed with $W \in \cW$. Then for every $q \in \Delta_\cW$, $g^{-1}(q)$ is contained in an element  $W$ of $\cW$ corresponding to a vertex of least dimensional simplex of $\Delta_\cW$ containing $q$. So for each $y \in Y$, $g^{-1}(q)\cap \pi^{-1}(y) \subseteq W\cap \pi^{-1}(y) \subseteq U $ for some $U$ in $\cU$. Choose a topological embedding $h: \Delta_\cW \to P$ for some polyhedron $P$ with $\dim P=k$. Then the map $f:=h\circ g$ is what we need.
\end{proof}

From Lemma \ref{conditional bridge}, we see that $\varphi$ is sub-additive and hence $\varphi$
  satisfies the conditions of Lemma \ref{OW}. Thus the limit $\lim_F \cD(\cU^F|Y)/|F|$ exists.

\begin{definition}
We define the {\it conditional mean topological dimension of $\Gamma \curvearrowright X$ relative to $\Gamma \curvearrowright Y$}  as
$$\mdim(X|Y):=\sup_\cU \lim_F \frac{\cD(\cU^F|Y)}{|F|}$$
for $\cU$ running over all finite open covers of $X$.
For simplicity, we may also say $\mdim(X|Y)$ is the  conditional mean dimension of $X$ relative to $Y$.
\end{definition}

 When $Y$ is a singleton, $\mdim(X|Y)$ recovers the mean topological dimension of $\Gamma \curvearrowright X$, which we denote by $\mdim(X)$ (see \cite[Definition 2.6]{LW00}). Moreover, as $\Gamma =\{e_\Gamma\}$ is the trivial group, $\mdim(X)$ recovers the {\it (covering) dimension} of $X$, which we denote by $\dim (X)$.

\begin{proposition} \label{generalization}

Let $\Gamma \curvearrowright Y$ and $\Gamma \curvearrowright Z$ be two dynamical systems. Let $\Gamma$ act on $ Y\times Z$ diagonally and $\pi: Y\times Z \to Y$ the projection map. Then
$\mdim(Y\times Z|Y)=\mdim(Z)$.
\end{proposition}

\begin{proof}
Fix a finite open cover $\cU$ of $Z$. To show $\mdim(Z) \leq \mdim(Y\times Z|Y)$, for any $F \in \cF(\Gamma)$, it suffices to show $\cD(\cU^F) \leq \cD(\cV^F|Y)$ for $\cV:=\{Y\times U: U \in \cU\}$.

Suppose that $\cD(\cV^F|Y)=\ord \ (\cW)$ for some finite open cover $\cW$ of $Y\times Z$ such that $\cW \vee \{\pi^{-1}(y)\}_{y \in Y}$ refines $\cV^F$. Consider the topological embedding $\varphi: Z \to Y\times Z$ sending $z$ to $(y_0, z)$ for some fixed $y_0$ in $Y$. Then for every $W \in \cW$, there exists  $U \in \cU^F$ such that $W\cap \pi^{-1}(y_0) \subseteq Y \times U$. It follows that $\varphi^{-1}(W) \subseteq U$ and hence $\varphi^{-1}(\cW)$ refines $\cU^F$. Thus
$$\cD(\cU^F) \leq \ord \ (\varphi^{-1}(\cW)) \leq \ord \ (\cW) =\cD(\cV^F|Y).$$

To show the other direction, for any finite open covers $\cU_0$ and $\cV_0$ of $Y$ and $Z$ respectively, we need only to show $\cD((\cU_0\times \cV_0)^F|Y) \leq \cD(\cV_0^F)$ for all $F \in \cF(\Gamma)$.

Let $\cD(\cV_0^F)=\ord \ (\cV)$ for some finite open cover $\cV$ of $Z$ refining $\cV_0^F$. Denote by $p_Z$ the projection map from $Y\times Z$ onto $Z$. Then for any $y \in Y$ and $V \in \cV$, one has $p_Z^{-1}(V)\cap \pi^{-1}(y)=\{y\}\times V \subseteq U\times V$ for any $U \in \cU_0^F$ containing $y$.  That means $p_Z^{-1}(\cV) \cap \{\pi^{-1}(y)\}_{y \in Y}$ refines $(\cU_0\times \cV_0)^F$. Thus
$$\cD((\cU_0\times \cV_0)^F|Y)  \leq \ord \ (p_Z^{-1}(\cV))=\ord \ (\cV) = \cD(\cV_0^F).$$
\end{proof}

Now we introduce a metric approach to the conditional mean dimension in line with \cite[Theorem 6.5.4]{Coornaert15B}. Let $\pi\colon X \to Y$ be a factor map and $\rho$ a compatible metric on $X$.  For any $\varepsilon >0$, denote by $\Wdim_\varepsilon(X|Y, \rho)$ the minimal dimension of a polyhedron $P$ which admits a continuous map $f\colon X \to P$ such that ${\rm diam} (f^{-1}(p)\cap \pi^{-1}(y), \rho) < \varepsilon $ for every $(p, y) \in P \times Y$. We call such a map  a {\it $(\rho, \varepsilon)$-embedding relative to Y}. For every $F \in \cF(\Gamma)$, denote by $\rho_F$ the metric  on $X$ defined as
$$\rho_F(x, y):=\max_{s \in F} \rho(sx, sy).$$
Then it is easy to check that the function  $ \cF(\Gamma)\cup \{\emptyset\} \to \bR$ sending $F$ to $\Wdim_\varepsilon(X|Y, \rho_F)$ satisfies the conditions of Lemma \ref{OW}. Thus the limit $\lim_F \frac{\Wdim_\varepsilon(X|Y, \rho_F)}{|F|}$ exists.

\begin{proposition} \label{metric approach}
For a compatible metric $\rho$ on $X$, we have
$$\mdim(X|Y)=\sup_{\varepsilon >0} \lim_F \frac{\Wdim_\varepsilon(X|Y, \rho_F)}{|F|}.$$
\end{proposition}

\begin{proof}
For the direction $``\leq "$, fix a finite open cover $\cU$ of $X$. Picking a Lebesgue number $\lambda$ of $\cU$ with respect to $\rho$, it suffices to show
$$\cD(\cU^F|Y) \leq \Wdim_{\lambda}(X|Y, \rho_F)$$
for every $F \in \cF(\Gamma)$. Let $f\colon X \to P$ be a continuous map with $\dim (P)= \Wdim_{\lambda}(X|Y, \rho_F)$ such that ${\rm diam} (f^{-1}(p)\cap \pi^{-1}(y), \rho_F) < \lambda$ for every $(p, y) \in P\times Y$. By choice of $\lambda$,  we have that $\{f^{-1}(p)\cap \pi^{-1}(y)\}_{(p, y)}$ refines $\cU^F$.  Applying Lemma \ref{conditional bridge} to $\cU^F$, it follows that $\cD(\cU^F|Y) \leq \dim (P) =\Wdim_{\lambda}(X|Y, \rho_F)$.

Now we show the converse direction of the equality. Fix $\varepsilon >0$ and pick a finite open cover  $\cU$ of $X$ consisting of some open sets of the diameter less than $\varepsilon$ under the metric $\rho$. It reduces to show
$$\Wdim_{\varepsilon}(X|Y, \rho_F) \leq \cD(\cU^F|Y)$$
for each $F \in \cF(\Gamma)$.  Applying Lemma \ref{conditional bridge} to $\cU^F$, we have a continuous map $f\colon X \to P$ with $\dim (P) =\cD(\cU^F|Y)$ such that $\{f^{-1}(p)\cap \pi^{-1}(y)\}_{(p, y)}$ refines $\cU^F$. By choice of $\cU$, we see that $f$ is a $(\rho_F, \varepsilon)$-embedding relative to $Y$. Thus
$$\Wdim_{\varepsilon}(X|Y, \rho_F) \leq \dim (P)= \cD(\cU^F|Y).$$

\end{proof}

\subsection{$G$-extensions}
Let us compute the conditional mean dimension of  $G$-extensions.
\begin{definition}{\cite[Page 411]{Bowen71}} \label{G-extension}
Let $\pi\colon X \to Y$ be a factor map and $\Gamma \curvearrowright G$ be another dynamical system. $X$  is called a {\it $G$-extension of $Y$} if there exists a continuous map $ X \times G \to X$ sending $(x, g)$ to $xg$ such that for any $x\in X, g, g' \in G$ and $t \in \Gamma$, we have
\begin{enumerate}
    \item $\pi^{-1}(\pi(x))=xG$;
    \item $xg=xg'$ exactly when $g=g'$;
    \item $t(xg)=(tx)(tg)$.
\end{enumerate}
\end{definition}
Note that  when $G$ is a group and the action  $\Gamma \curvearrowright G$ is trivial, the factor map $\pi$ recovers as a principal group extension.

\begin{example}
Let $\Gamma \curvearrowright Y$ and $\Gamma \curvearrowright G$ be two dynamical systems such that $G$ is a compact group and $\Gamma$ acts on $G$ by continuous automorphisms.  A (continuous) {\it cocycle} is a continuous map $\sigma\colon \Gamma \times Y \to G$ such that
$$\sigma(st, y)=\sigma(s, ty)\cdot s(\sigma(t, y))$$
for every $s, t \in \Gamma$ and $y \in Y$.
It induces an action of $\Gamma$ on  $Y \times G$ by
$$s(y, g):=(sy, \sigma(s,y)\cdot (sg))$$
for all $s \in \Gamma, y \in Y$ and $g\in G$.  Then $Y\times G$ is a $G$-extension of $Y$ in light of the map $(Y \times G) \times G \to Y\times G$ sending $((y, g), h)$ to $(y, gh)$. We denote by $Y\times_\sigma G$ the $G$-extension from such a cocycle $\sigma$.

Another source of $G$-extension arise when the underlying systems have group structure. Recall that a dynamical system $\Gamma \curvearrowright X$ is called an {\it algebraic action} if $X$ is a compact metrizable group and the action of $\Gamma$ on $X$ is by continuous automorphisms.  Let $\pi\colon X \to Y$ be a factor map between algebraic actions such that $\pi$ is a group homomorphism. Put $G=\ker (\pi)$. Then $X$ is a $G$-extension given by sending $(x, g) \in X \times G$ to $xg$.
\end{example}

\begin{proposition} \label{group extension}
Let $X$ be a $G$-extension of $Y$ for some compact metrizable space $G$. Then $\mdim(X|Y)\geq \mdim(G)$. If $\pi\colon X \to Y$ admits a continuous {\it section} $\tau \colon Y \to X$ in the sense that $\tau$ is continuous such that $\pi \circ \tau=\id_Y$, we have $\mdim(X|Y)= \mdim(G)$.
\end{proposition}

\begin{proof}
Fix $F \in \cF(\Gamma)$. Let  $\rho_X, \rho_G$ be two compatible metrics on $X$ and $G$ respectively.

Pick a point  $x_0$ from $X$.
By definition of $G$-extension, for any $x \in G$ and $g, g' \in G$, $xg=xg'$ exactly when $g=g'$. Thus by compactness of $X$ and $G$, for any $\varepsilon > 0$ there exists  $\delta > 0$ satisfying the following property: for any $g, g' \in G$ such that $\rho_X(xg, xg') < \delta$ for some $x \in X$, we have
\begin{align} \label{1}
   \rho_G(g, g') < \varepsilon.
\end{align}
 Let $\psi\colon X \to P$ be a $(\rho_{X, F}, \delta)$-embedding relative to $\pi$. Then for any $g, g' \in G$ with $\psi(x_0g)=\psi(x_0g')$, since $\pi(x_0g)=\pi(x_0)=\pi(x_0g')$, we have $\rho_{X, F}(x_0g, x_0g') < \delta$.
By inequality (\ref{1}), we obtain $\rho_{G, F}(g, g') < \varepsilon$. Denote by $\varphi$ the map $G \to X$ sending $g$ to $x_0g$. This concludes that the map $\psi \circ \varphi$ is a $(\rho_{G, F}, \varepsilon)$-embedding. The desired inequality then follows from a limit argument.

Now we assume that $\pi$ admits a continuous cross section $\tau\colon Y \to X$. For any $\varepsilon > 0$, there exists $\delta > 0$ such that
$$\rho_X(xg, xg') < \varepsilon$$
for any $x \in X$ and $g, g' \in G$ such that $\rho_G(g, g') < \delta$.

 Assume that $\varphi\colon G \to Q$ is a $(\rho_{G, F}, \delta)$-embedding for some polyhedron $Q$. For each $x \in X$, since $\pi ( \tau (\pi(x)))=\pi(x)$, we have $x, \tau(\pi(x)) \in \pi^{-1}(\pi(x))=xG$ and hence $x=\tau(\pi(x))g_x$ for a unique $g_x \in G$. Now define $\psi\colon X \to Q$ by sending $x$ to $\varphi(g_x)$. Then the continuity of $\psi$ is guaranteed by the continuity of $\tau$. For any $x, x' \in X$ with the same image under $\pi$ and $\psi$, since $\varphi$ is a $(\rho_{G, F}, \delta)$-embedding, we have $\rho_G(sg_x, sg_{x'})< \delta$ for any $s \in F$. By the design of $\delta$, we have
\begin{align*}
    \rho_X(sx, sx')&=\rho_X(s(\tau(\pi(x))g_x), s(\tau(\pi(x'))g_{x'})) \\
                   &=\rho_X((s(\tau(\pi(x))))(sg_x), (s(\tau(\pi(x))))(sg_{x'}))) <\varepsilon.
\end{align*}
That implies that $\Wdim_\varepsilon(X|Y, \rho_{X, F}) \leq \dim (Q)$. The inequality then follows by running some limit argument.
\end{proof}

\begin{example}
Let $\bZ\Gamma$ be the integral group ring of $\Gamma$ and $f \in \bZ\Gamma$ (see \cite{S95} for more details about group rings). Consider that $\Gamma$ acts on $(\bR/\bZ)^\Gamma$ by left shift. Let $R(f)\colon X:=(\bR/\bZ)^\Gamma \to (\bR/\bZ)^\Gamma$ be the group homomorphism sending $x$ to $xf$. Set $G:=\ker (R(f))$. Then  the induced factor map $\pi_f\colon X \to Y:=\Ima(R(f))$ shows that $X$ is a $G$-extension of $Y$.
Suppose that $fuf=f$ for some $u \in \bZ \Gamma$. Then $\pi_f$ admits a continuous section $Y \to X$ by sending $y$ to $yu$. From Proposition \ref{group extension}, we have $\mdim(X|Y)=\mdim(G)$.
\end{example}

\subsection{Embedding problem}
By a {\it dynamical emmbedding} $\varphi \colon X \to Y$ between two dynamical systems $\Gamma \curvearrowright X$ and $\Gamma \curvearrowright Y$, we mean $\varphi$ is a continuous injective map such that $\varphi(sx)=s\varphi(x)$ for every $s \in \Gamma$ and $x \in X$. With the help of conditional mean dimension, we can strengthen some results regarding the dynamical embedding problem in terms  of Rokhlin dimension \cite[Theorem 3.1]{GQS18}.

The notion of Rokhlin dimension appears in the classification of $C^*$-algebras induced from topological dynamical systems. The definition in the setting of topological dynamical systems is due to Winter, explicitly formulated by Szabo for $\bZ^k$-actions \cite[Definition 2.1]{S15}, and extended to the action of residually finite groups by Szabó, Wu, and Zacharias \cite{SWZ19}. In particular, Rokhlin dimension of infinite finitely generated nilpotent group actions is estimated in \cite[Corollary 8.5]{SWZ19}. 

Now we can consider the definition for the action of amenable groups.
\begin{definition}
Let $\Gamma \curvearrowright X$ be a continuous action by a countable amenable group $\Gamma$ on a compact metrizable space $X$. We say $\Gamma \curvearrowright X$ has {\it Rokhlin dimension $d$}, denoted as 
$$\dim_{\rm Rok}(X, \Gamma)=d,$$
if $d$ is the smallest nonnegative integer such that for every finite subset $K$ of $\Gamma$ and every $\delta >0$, there exists $(d+1)$ subsets $F_0, \cdots, F_d \in \cB(K, \delta)$ and $(d+1)$ open sets $U_0,\cdots, U_d$ satisfying that:\\
$(i)$ the subsets $\{s\overline{U_i}\}_{s \in F_i}$ are pairwise disjoint for every $i=0,\cdots, d$;\\
$(ii)$ the union $\cup_{i=0}^d \sqcup_{s \in F_i} sU_i$ covers the whole space $X$.
\end{definition}

Observe that this definition allows the distinct towers $F_iU_i$'s  to overlap and  Rokhlin dimension does not increase when passing to the extension systems. 

With the help of the notion of conditional mean dimension, we can improve the statement of  \cite[Theorem 3.1]{GQS18} in the following. The proof of  \cite[Theorem 3.1]{GQS18} works here by applying  Lemma \ref{conditional embedding approx} as a conditional version of \cite[Lemma 2.1]{GT14}.
\begin{theorem}\label{conditional embedding}
 Let $D$ be a nonnegative integer and $L$ a positive integer. Suppose that $\pi \colon X \to Y$ is a factor map with $\dim_{{\rm Rok}}(Y, \Gamma)=D$ and $\mdim(X|Y) < L/2$. Then there exists a dynamical embedding from $X$ to $\left( ([0,1]^{(D+1)L})^\Gamma  \right) \times Y$ where the later is endowed with the product action from the shift action on $([0,1]^{(D+1)L})^\Gamma$ and the action $\Gamma \curvearrowright Y$.  
\end{theorem}

The following lemma is a conditional version of \cite[Lemma 2.1]{GT14}, whose proof works here. 
\begin{lemma} \label{conditional embedding approx}
Let $\pi \colon X \to Y$ be a continuous map between compact metrizable space with a compatible metric $\rho$ on $X$.  Suppose that $f\colon X \to [0, 1]^L$ is a continuous map such that $||f(x)-f(x')||_\infty < \delta$ for every $x, x' \in X$ with $\rho(x, x') < \varepsilon$. Assume that $\Wdim_\varepsilon(X|Y, \rho) < L/2$. Then there exists a $(\rho,\varepsilon)$-embedding $g \colon X \to [0, 1]^L$ relative to $Y$ satisfying that 
$\sup_{x \in X} ||f(x)-g(x)||_\infty < \delta.$
\end{lemma}

Combining Theorem \ref{conditional embedding} with Proposition \ref{group extension}, we obtain the following dynamical embedding result for $G$-extensions.

\begin{corollary} \label{G-extension embedding}
Let $X$ be a $G$-extension of $Y$ for some compact metrizable space $G$. Suppose that the factor map $\pi \colon X \to Y$ admits a continuous section. Assume that $\mdim(G) < L/2$ and $\dim_{\rm Rok}(Y, \Gamma)=D$ for some positive integer $L$ and nonnegative integer $D$. Then there exists a dynamical embedding from $X$ to $([0,1]^{(D+1)L})^\Gamma \times Y$.
\end{corollary}

\subsection{Mean dimension of fibers}

Given a finite open cover $\cU$ of $X$, for any closed subset $K$ of $X$, denote by $\cU|_K$ the finite open cover of $K$ restricted from $\cU$, i.e. $\cU|_K=\{U\cap K: U \in \cU\}$. Taking advantage of $\Gamma$-invariance of $X$,  we can similarly consider the mean dimension of $K$ as \cite[Section 1.5]{Gromov99T} and Tsukamoto \cite[Remark 4.7]{Tsukamoto08}.

\begin{definition} \label{mdim def of closed sets}
Fix a {\it F{\o}lner sequence} $\cF:=\{F_n\}_{n\geq 1}$ of $\Gamma$, i.e. for any $s \in \Gamma$, $|sF_n\Delta F_n|/|F_n|$ converges to $0$ as $n$ goes to the infinity.  We define the {\it mean dimension of $K$ }  as
$$\mdim(K, \Gamma):=\sup_\cU \varliminf_{n \to \infty} \frac{\cD(\cU^{F_n}|_K)}{|F_n|},  $$
where $\cU$ ranges over all finite open covers of $X$.

\end{definition}

By the same  argument of Proposition \ref{metric approach},  we have

\begin{proposition} \label{restricted metric approach}
Fix a compatible metric $\rho$ on $X$. We have
$$\mdim(K, \Gamma)=\sup_{\varepsilon > 0} \varliminf_{n \to \infty} \frac{\Wdim_\varepsilon(K, \rho_{F_n})}{|F_n|}.$$
\end{proposition}

In particular, considering the fibers of a factor map $\pi\colon X \to Y$, in light of metric approach formulas in Propositions \ref{metric approach} and \ref{restricted metric approach},  we have the following estimation.

\begin{proposition} \label{fiber as lower bound}
For every $y \in Y$, we have $\mdim(\pi^{-1}(y), \Gamma) \leq \mdim(X|Y)$.
\end{proposition}

By a modified argument of Proposition \ref{group extension}, we have
\begin{proposition} \label{fiber inequality}
Let $ X$ be a $G$-extension of $Y$. Then $\mdim(G) = \mdim(\pi^{-1}(y), \Gamma)$ for every $y \in Y$.
\end{proposition}

 We have a satisfactory answer to Question \ref{main interest} in the following case.
\begin{corollary} \label{answer for algebraic actions}
Let $\pi\colon X \to Y$ be a factor map of algebraic actions such that $\pi$ is a group homomorphism. Write $G:=\ker (\pi)$. Then  $\mdim(G)=\mdim(\pi^{-1}(y), \Gamma)$ for every $y \in Y$. In particular, we have
$$\mdim(X) = \mdim(Y)+\sup_{y \in Y} \mdim(\pi^{-1}(y), \Gamma).$$
\end{corollary}

\begin{proof}
Clearly $X$ is a $G$-extension of $Y$. Thus the first statement is true from Proposition \ref{fiber inequality}.  By the addition formula for mean dimension of algebraic actions \cite[Corollary 6.1]{LL15}, we have
\begin{align*}
    \mdim(X) &=\mdim(Y)+\mdim(G)\\
    & =\mdim(Y)+\sup_{y \in Y} \mdim(\pi^{-1}(y), \Gamma).
\end{align*}
\end{proof}

\section{Proof of Theorem \ref{subadditivity}}

In this section, we give the proof of Theorem \ref{subadditivity}.

First, we recall the quasi-tiling lemma of amenable groups as follows \cite[Page 24,Theorem 6]{OW87} \cite[Theorem 8.3]{Li12}. In fact, one can require all quasi-tiles contain the identity element $e_\Gamma$ of $\Gamma$.
 Let $\varepsilon > 0$ and $F_1, \cdots, F_m \in \cF(\Gamma)$, we say $\{F_j\}_{j=1}^m$ are {\it $\varepsilon$-disjoint} if there exists $F_j' \subseteq F_j$ for every $j=1, \cdots, m$ such that $\{F_j\}_{j=1}^m$ are pairwise disjoint and $|F_j'| \geq (1-\varepsilon)|F_j|$ for every $j=1, \cdots, m$.
\begin{lemma} \label{quasi-tiling}
Let $\varepsilon > 0$ and $K \in \cF(\Gamma)$. Then there exists $\delta > 0$ and $K', F_1, \cdots, F_m \in \cF(\Gamma)$ such that
\begin{enumerate}
    \item  $e_\Gamma \in F_j \in \cB(K, \varepsilon)$, for all $j=1,\cdots, m$;
\item  For each $A \in \cB(K', \delta)$, there exist $D_1, \cdots, D_m \in \cF(\Gamma)$ such that the family $\{F_jc: c \in D_j, j=1,\cdots, m\}$ are $\varepsilon$-disjoint subsets of $A$, and $|A\setminus \bigcup_{j=1}^m F_j D_j| \leq \varepsilon|A|$.
\end{enumerate}
\end{lemma}
We call those $F_j$'s {\it quasitiles} of $\Gamma$ and $D_j$'s the {\it tiling centers} of $A$.

\begin{proof} [Proof of Theorem \ref{subadditivity}]
Fix a finite open cover $\cU$ of $X$. Let $ 0 < \varepsilon < 1$ and $K \in \cF(\Gamma)$.  By Lemma \ref{quasi-tiling}, there exist $\delta > 0$, $K' \in \cF(\Gamma)$, and tiles $F_1, \cdots, F_m \in \cF(\Gamma)$,  such that each $A\in \cB(K', \delta) $ admits tiling centers $D_1, \cdots, D_m \in \cF(\Gamma)$ satisfying the conditions in Lemma \ref{quasi-tiling}.

For each $j=1, \cdots, m$ choose a finite open cover $\cW_j$ of $X$ such that $\ord \ (\cW_j) = \cD(\cU^{F_j}|Y)$ and $\cW_j \vee \{\pi^{-1}(y)\}_{y \in Y}$ refines $\cU^{F_j}$. Without loss of generality, we may assume $\cW_j' \vee \{\pi^{-1}(y)\}_y$ still refines $\cU^{F_j}$ for $\cW_j':=\{\overline{W}\}_{W \in \cW_j}$. Then for each $y \in Y$ and $W \in \cW_j$ there exists $U \in \cU^{F_j}$ and an open neighborhood $V_{y, W}$ of $y$ such that
$$\overline{W} \cap \pi^{-1}(V_{y, W}) \subseteq U.$$
By compactness there exists a subfamily $\cV_j$ of $\{ \cap_{W \in \cW_j} V_{y, W}: y \in Y\}$ such that $\cV_j$ still makes an open cover of $Y$. Clearly  $\cW_j \vee \pi^{-1}(\cV_j)$ refines $\cU^{F_j}$.
Put $\cV=\vee_{j=1}^m \cV_j$ (depending only on $\cU$ and $K$). It follows that $\cW_j \vee \pi^{-1}(\cV)$ refines $\cU^{F_j}$ for every $j=1, \cdots, m$.

Now for $A \in \cB(K', \delta)$, choose a finite open cover $\cW_A$ of $Y$ such that $\ord \ (\cW_A) = \cD(\cV^A)$ and $\cW_A$ refines $\cV^A$. Since $F_j$ contains the identity of $\Gamma$, we have $r\cW_A$ refines $\cV$ for every $r \in D_j$ and $j=1, \cdots, m$. By construction of $\cV$, we have $\cW_j \vee \pi^{-1}(r\cW_A)$ refines $\cU^{F_j}$. Hence $(\vee_{j=1}^m \vee_{r \in D_j} r^{-1}\cW_j) \vee \pi^{-1}(\cW_A)=\vee_{j=1}^m \vee_{r \in D_j} r^{-1}(\cW_j \vee \pi^{-1}(r\cW_A))$ refines $\cU^{\cup_j F_jD_j}$.Thus
\begin{align*}
    \cD(\cU^A) & \leq \cD(\cU^{\cup_j F_jD_j})+ \cD(\cU^{A\setminus \cup_j F_jD_j})\\
   & \leq \ord \ ((\vee_{j=1}^m \vee_{r \in D_j} r^{-1}\cW_j) \vee \pi^{-1}(\cW_A)) +\varepsilon |A|\cD(\cU)\\
   & \leq \cD(\cV^A) + \sum_j |D_j|\cD(\cU^{F_j}|Y)+\varepsilon |A|\cD(\cU).
\end{align*}

Since $\{F_jc\}_{j, c}$ are $\varepsilon$-disjoint subsets of $A$, we have
$$\sum_{j=1}^m |F_j||D_j| \leq \frac{|A|}{1-\varepsilon}.$$
It follows that
$$\sum_{j=1}^m |D_j| \cD(\cU^{F_j}|Y) =\sum_{j=1}^m |F_j||D_j| \frac{\cD(\cU^{F_j}|Y)}{|F_j|}
 \leq \frac{|A|}{1-\varepsilon} \sup_{F \in \cB(K, \varepsilon)}  \frac{\cD(\cU^{F}|Y)}{|F|} .$$
 So
 $$\frac{ \cD(\cU^A)}{|A|} \leq \frac{\cD(\cV^A)}{|A|} + \frac{1}{1-\varepsilon}\sup_{F \in \cB(K, \varepsilon)}  \frac{\cD(\cU^{F}|Y)}{|F|} +\varepsilon \cD(\cU).$$
 Since $A \in \cB(K', \delta)$ is arbitrary, we get
 $$\mdim(\cU) \leq \mdim(Y) +\frac{1}{1-\varepsilon}\sup_{F \in \cB(K, \varepsilon)}  \frac{\cD(\cU^{F}|Y)}{|F|} +\varepsilon \cD(\cU)$$
 for $\mdim(\cU):=\lim_F \frac{\cD(\cU^F)}{|F|}$.
 Taking the limit for $K$ and $\varepsilon$, we have
 $$\mdim(\cU) \leq \mdim(Y)+\mdim(X|Y).$$
 Since $\cU$ is arbitrary, this completes the proof.
\end{proof}

\begin{remark}
There is a number of reasons why  the converse inequality in Theorem \ref{subadditivity} can fail.   It is well known that  Cantor set can continuously map onto any compact metrizable space. In particular, for the action of trivial group, we have the converse inequality in Theorem \ref{subadditivity} fails for such a surjective map. Moreover, Boltyanski\u{\i} constructed an example of a compact metrizable space $X$ such that $\dim (X\times X) < 2\dim(X)$  (See \cite{Boltyanski51}). As a consequence, we know that the converse of inequality in Theorem \ref{subadditivity} can fail even for a projection map.
\end{remark}

\begin{corollary}
Let $\Gamma \curvearrowright Y$ be a dynamical system and $\Gamma \curvearrowright G$ an algebraic action. Suppose that
$\sigma: \Gamma \times Y \to G$ is a continuous cocycle. Then for the induced $G$-extension $\Gamma \curvearrowright Y\times_\sigma G$, we have
$$\mdim(Y\times_\sigma G)\leq \mdim(Y)+\sup_{y \in Y} \mdim(\pi^{-1}(y), \Gamma).$$
\end{corollary}

\begin{proof}
Clearly  $Y\times_\sigma G$ admits a continuous cross section.  Thus by  Propositions  \ref{group extension} and \ref{fiber inequality}, we have $\mdim(Y\times_\sigma G|Y)=\mdim(G)=\mdim(\pi^{-1}(y), \Gamma)$ for every $y \in Y$. Applying Theorem \ref{subadditivity}, we have
\begin{align*}
    \mdim(Y\times_\sigma G)&\leq \mdim(Y)+\mdim(Y\times_\sigma G|Y)\\
                           &=\mdim(Y)+\sup_{y \in Y} \mdim(\pi^{-1}(y), \Gamma).
\end{align*}
\end{proof}

\section{conditional metric mean dimension}

In contrast with metric mean dimension, it is natural to consider its conditional version.
For a metrizable space $X$ with a compatible metric $\rho$ and $\varepsilon > 0$, a subset $E \subseteq X$ is called {\it $(\rho, \varepsilon)$-separating} if $\rho(x, x') \geq \varepsilon$ for every distinct $x, x' \in E$. Denote by $N_\varepsilon(X, \rho)$ the maximal cardinality of $(\rho, \varepsilon)$-separating subsets of $X$.

\begin{definition}
Let $\Gamma \curvearrowright X$ be a dynamical system. Fix a compatible metric $\rho$ on $X$. Set
$$N_\varepsilon(X|Y, \rho)=\max_{y \in Y} N_\varepsilon(\pi^{-1}(y), \rho).$$
We define the {\it conditional metric mean dimension of $\Gamma \curvearrowright (X, \rho)$ relative to $\Gamma \curvearrowright Y$} as
$$\mdim_\rM(X|Y, \rho):=\varliminf_{\varepsilon \to 0} \varlimsup_F \frac{\log N_\varepsilon(X|Y, \rho_F)}{|\log\varepsilon||F|}.$$
\end{definition}
Again, when $Y$ is a singleton, $\mdim_\rM(X|Y, \rho)$ recovers as the metric mean dimension of $\Gamma \curvearrowright (X, \rho)$, which we
denote by $\mdim_\rM(X, \rho)$ (see \cite[Definition 4.1]{LW00}).

\begin{remark}
Recall that the {\it mesh} of a finite open cover $\cU$ for $(X, \rho)$ is defined by
$${\rm mesh}(\cU, \rho):=\max_{U \in \cU} {\rm diam} (U, \rho).$$
In terms of this quantity, one can also give an equivalent definition of conditional metric mean dimension by considering the function $\cF(\Gamma) \to \bR$ sending $F$ to
$$\log \max_{y \in Y} \min_{{\rm mesh}(\cU_y, \rho_F)< \varepsilon} |\cU_y|$$
for $\cU_y$ ranging over all finite open covers of $\pi^{-1}(y)$.
It is easy to check that this function satisfies the conditions of Lemma \ref{OW}.
\end{remark}

\begin{proposition}
Let $ X$ be a $G$-extension of $ Y$ for some compact metrizable space $G$. Suppose that $\rho_X$ and $\rho_G$ are two compatible metrics on $X$ and $G$ respectively such that
$$\rho_X(xg, xg') = \rho_G(g, g')$$
for all $x \in X$ and $g, g' \in G$. Then
$$\mdim_\rM(X|Y, \rho_X) = \mdim_\rM(G, \rho_G).$$
\end{proposition}

\begin{proof}
By definition, a subset $E$ of $\pi^{-1}(y)$ is of a form $x_0G_0$ for some $x_0 \in X$ and $G_0 \subseteq G$. Then $E$ is $(\rho_F, \varepsilon)$-separating subset of $\pi^{-1}(y)$ if and only if $G_0$ is a $(\rho_F, \varepsilon)$-separating subset of $G$.
By a limit argument, we have
$\mdim_\rM(X|Y, \rho_X) \leq \mdim_M(G, \rho_G)$.

To see  the converse of equality, it suffices to notice that if a subset $G_0$ of $G$ is $(\rho_F, \varepsilon)$-separating, then for every $x \in X$,  $xE$ is a $(\rho_F, \varepsilon)$-separating subset of $\pi^{-1}(\pi(x))$.
\end{proof}

\begin{question} \label{comparison}
For a factor map $\pi \colon X \to Y$, is it true that $\mdim(X|Y) \leq \mdim_\rM(X|Y, \rho)$ for every compatible metric $\rho$ on $X$?
\end{question}

\section{Mean dimension given a measure}

In this section, we define the mean dimension given a measure on the factor system and discuss its properties. We start with
a key lemma.
\begin{lemma} \label{USC}
Suppose that $\varphi\colon X \to Y$ is a continuous map between compact metrizable spaces and $\cU$ is a finite open cover of $X$. Then the map $Y \to \bR$ sending $y$ to $\cD(\cU|_{\varphi^{-1}(y)})$ is upper semicontinuous.
\end{lemma}

\begin{proof}
Fix $y \in Y$ and put $\cD(\cU|_{\varphi^{-1}(y)})=d$. By definition, there exists a finite open cover $\cV$ of $X$ such that $\cV|_{\varphi^{-1}(y)} \succ \cU|_{\varphi^{-1}(y)}$ and $\ord(\cV|_{\varphi^{-1}(y)})=d$. Write $\cV$ as $\cV=\{V_i\}_{i \in I}$. Then for any $J \subseteq I$ such that $|J|> d+1$, one has
$$(\cap_{j \in J}V_j)\cap \varphi^{-1}(y)=\emptyset.$$

Since $X$ is normal, there exists a finite open cover $\cV'=\{V_i'\}_{i \in I}$ such that  $\overline{V_i'} \subseteq V_i$ for all $i \in I$ (see \cite[Corollary 1.6.4]{Coornaert15B}). In particular, $\varphi^{-1}(y)$ has the empty intersection with $ \cap_{j \in J} \overline{V_j'}$ for all $J \subseteq I$ such that $|J|> d+1$. Thus we conclude that $y$ sits inside the open set
$$Y \setminus \varphi(\cap_{j \in J} \overline{V_j'})=\{z \in Y: \varphi^{-1}(z) \subseteq (\cap_{j \in J} \overline{V_j'})^c \}$$
for each $J \subseteq I$ with $|J|> d+1$.  That means, as $z$ approaches to $y$, $\varphi^{-1}(z)$ has the empty intersection with $ \cap_{j \in J} V_j'$ for every $J \subseteq I$ with $|J|> d+1$. So by definition, $\cD(\cU|_{\varphi^{-1}(z)}) \leq \ord(\cV'|_{\varphi^{-1}(z)}) \leq d$. This finishes the proof.
\end{proof}

Based on this lemma, we are safe to define the measure-theoretic conditional mean dimension.

\begin{definition}
Denote by $M_\Gamma(Y)$ the collection of $\Gamma$-invariant Borel probability measures on $Y$. For any $\nu \in M_\Gamma(Y)$, set
$$\cD(\cU|\nu):=\int_Y \cD(\cU |_{ \pi^{-1}(y)}) d\nu(y).$$
Note that $\cD(\cU^{Fs} |_{ \pi^{-1}(y)})=\cD(\cU^F |_{ \pi^{-1}(sy)})$ for any $s \in \Gamma$.
It follows that the function $\cF(\Gamma) \cup \{\emptyset\} \to \bR$  sending $F$ to $\cD(\cU^F|\nu)$ satisfies the conditions of Lemma \ref{OW}. We define the {\it mean dimension of $\Gamma \curvearrowright X$ given $\nu$} as
$$\mdim(X|\nu):=\sup_\cU \lim_F \frac{\cD(\cU^F|\nu)}{|F|},$$
for $\cU$ ranging over all finite open covers of $X$.
\end{definition}

\begin{example}
In the same setting of Proposition \ref{generalization}, it is easy to see that
$\mdim(Y\times Z|\nu)=\mdim(Z)$ for any $\nu \in M_\Gamma(Y)$.
\end{example}

Following the similar argument as in the proof of \cite[Lemma 6.8]{Li12} by taking liminf instead, we have
\begin{proposition} \label{inequalities}
Let $\pi: X \to Y$ be a factor map. Then
$$\sup_{\nu \in M_\Gamma(Y)}\mdim(X|\nu) \leq \sup_{y \in Y} \mdim(\pi^{-1}(y), \Gamma).$$
\end{proposition}

Combining Proposition \ref{fiber as lower bound} with \ref{inequalities}, we see that conditional mean dimension given a measure serves as
a lower bound of conditional mean  dimension.

Recall that a finite subset $T$ of $\Gamma$ is called a {\it tile} if there exists a subset $C$ of $\Gamma$ such that $\{Tc\}_{c \in C}$ makes a partition of $\Gamma$. A F{\o}lner sequence $F_n$'s is called a {\it tiling F{\o}lner sequence} if each $F_n$  is a tile.  It is well known that all elementary amenable groups including abelian groups admit a tiling F{\o}lner sequence \cite{Weiss01}.
\begin{proposition} \label{equality}
When $\Gamma$ is an abelian group, we have
$$\sup_{\nu \in M_\Gamma(Y)}\mdim(X|\nu) = \sup_{y \in Y} \mdim(\pi^{-1}(y), \Gamma),$$
where $\mdim(\pi^{-1}(y), \Gamma)$ is defined along a tiling F{\o}lner sequence of $\Gamma$.
\end{proposition}

\begin{proof}
From Proposition \ref{inequalities},  We need only to prove the for every $y \in Y$ there exists a $\mu \in M_\Gamma(Y)$ such that $\mdim(\pi^{-1}(y), \Gamma) \leq \mdim(X|\mu)$.  

Fix a finite open cover $\cU$ of $X$. For every $F \in \cF(\Gamma)$ and $z\in Y$, set $f_F(z)=\cD(\cU^F|_{ \pi^{-1}(z)})$. By Lemma \ref{USC},  $f_F$ is upper semicontinuous. Note that \cite[Lemma 3.6]{LY}  holds when every $f_F$ is upper semicontinuous. Pick a cluster point of the measures $\frac{1}{|F_n|}\sum_{s \in F_n} \delta_{sy}$ under the weak*-topology, we have $\mu \in M_{\Gamma}(Y)$. Applying \cite[Lemma 3.6]{LY} to the measures $\nu_n=\delta_y$, it follows that
$$\varliminf_{n \to \infty} \frac{f_{F_n}(y)}{|F_n|} = \varliminf_{n \to \infty} \int_Y \frac{f_{F_n}}{|F_n|} d\nu_n \leq \lim_{n \to \infty} \int_Y \frac{f_{F_n}}{|F_n|} d\mu.$$
Therefore,
\begin{align*}
    \mdim(\pi^{-1}(y), \Gamma) &=\sup_{\cU} \varliminf_{n \to \infty} \frac{f_{F_n}(y)}{|F_n|}\\
    &\leq \sup_{\cU}  \lim_{n \to \infty} \int_Y \frac{f_{F_n}}{|F_n|} d\mu=\mdim(X|\mu). 
\end{align*}
Since $y$ is arbitrary, this finishes the proof.

\end{proof}


  \end{document}